\documentclass[12pt]{article}
\usepackage[utf8]{inputenc}
\usepackage{amsmath,amssymb,amsthm,amscd}
\usepackage[all,cmtip]{xy}
\usepackage{fullpage}
\usepackage{hyperref}
\usepackage{graphicx}
\usepackage{stackrel}
\usepackage{tikz} 
\usetikzlibrary{matrix,arrows}
\usepackage{color}
\usepackage{textcomp}

\newtheoremstyle{myremark} % name
    {7pt}                    % Space above
    {7pt}                    % Space below
    {}  	                 % Body font
    {}                           % Indent amount
    {\bf}       	         % Theorem head font
    {.}                          % Punctuation after theorem head
    {.5em}                       % Space after theorem head
    {}  % Theorem head spec (can be left empty, meaning ‘normal’)
    
\theoremstyle{plain}
\newtheorem{lemma}{Lemma}[section]

\newtheorem{corollary}[lemma]{Corollary}

\newtheorem{theorem}[lemma]{Theorem}
\newtheorem*{theorem-main}{Theorem~\ref{thm:main}}
\newtheorem*{theorem-secondary}{Theorem~\ref{thm:secondary}}
\theoremstyle{definition}

\newtheorem{definition}[lemma]{Definition}

\theoremstyle{myremark}
\newtheorem{remark}[lemma]{Remark}

\newcommand{\Z}{\ensuremath{\mathbb{Z}}}

\newcommand{\p}{{\mathfrak{p}}}
\newcommand{\m}{{\mathfrak{m}}}
\newcommand{\q}{{\mathfrak{q}}}

\newcommand{\sdim}{{\phantom{}^*\mathrm{dim}}}

\newcommand{\xra}{\xrightarrow}

\newcommand{\grmod}{\mathfrak{grmod}}

\newcommand{\vdim}{\mathrm{vdim}}

\begin{document}

\title{A Degree Formula for Equivariant Cohomology Rings}
\author{Mark Blumstein and Jeanne Duflot}
%\date{\today}

\maketitle
\begin{abstract}
This paper generalizes a result of Lynn on the ``degree" of an equivariant cohomology ring $H^*_G(X)$.  The degree of a graded module is a certain coefficient of its Poincar\'{e} series, and is closely related to multiplicity.  In the present paper, we study these commutative algebraic invariants for equivariant cohomology rings. The main theorem is an additivity formula for degree: $$\deg(H^*_G(X)) = \sum_{[A,c] \in \mathcal{Q'}_{max}(G,X)}\frac{1}{|W_G(A,c)|} \deg(H^*_{C_G(A,c)}(c)).$$ We also show how this formula relates to the additivity formula from commutative algebra, demonstrating both the algebraic and geometric character of the degree invariant. 
\end{abstract}
\setcounter{tocdepth}{1}
\tableofcontents

\section{Introduction}
In the 1950's and 60's equivariant cohomology developed in many arenas: Cartan's theory of equivaraint differential forms, representation theory, and Atiyah's $K$-theory, to name a few. However, the commutative algebraic properties of the Borel equivariant cohomology ring $H^*_G(X)$ associated to the action of a group $G$ on a space $X$, weren't well understood until 1971 when Quillen published \textit{The Spectrum of an Equivariant Cohomology Ring: I/II} (\cite{Quillen1},\cite{Quillen2}). Quillen explored the Krull dimension, prime spectrum, and localization of the equivariant cohomology ring $H^*_G(X)$, where $G$ is a compact Lie group, $X$ is a $G$-space satisfying certain hypotheses, and cohomology coefficients are taken in a field $k$ of characteristic $p$. He showed that tucked away in each of these commutative algebraic constructs is a great deal of information about the topology and geometry of the $G$-action on $X$. Building on Quillen's insights, many authors have researched commutative algebraic invariants of equivariant cohomology rings.  For example, more recent research  includes  \cite{Duflot}, \cite{Duflot3}, where the concepts of localization and primary decomposition are explored.  Symonds \cite{Symonds} studied Castelnuovo-Mumford regularity and the article of Lynn \cite{Lynn}  is a direct precursor for this paper.  For torus or elementary abelian group actions, the Cohen-Macaulay property from commutative algebra was studied in \cite{GoTo}, \cite{AlFrPu}, \cite{AFP}.  This is by no means an exhaustive list of references, but shows how the study of commutative-algebraic concepts and invariants in algebraic topology, begun by Quillen, continues to the present day.

The present paper  answers a problem posed by Lynn \cite{Lynn} about a certain algebraic invariant called the \textit{degree}.  The degree of a  finitely generated graded module $M$ over a  positively graded ring  $S$, finitely generated as an algebra over  $S_0$, with $S_0$ an Artinian ring,  is defined in terms of the Poincar\'{e} series of $M$ ($PS(M)$):
$$\deg(M) \doteq \lim_{t \rightarrow 1} (1-t)^{\dim_S(M)}PS(M);$$  here ``$\dim_S(M)$" is the Krull dimension of the $S$-module $M$.   The study of degree is ubiquitous in commutative algebra and algebraic geometry; indeed, for a ``standard" graded algebra $R$-in other words, if  $R$ is a finitely generated graded algebra over a field, generated by elements of graded degree 1--the degree of a finitely generated graded $R$-module $M$ is equal to its Samuel multiplicity.    Now, equivariant cohomology rings are usually not generated by elements in a single graded degree, and the degree of $M$ is not always an integer,  but there is still a relationship between the degree and an analog of the Samuel multiplicity for graded modules over graded rings:  for an exposition of the precise relationship between degree and this analog of the Samuel multiplicity in  graded rings, see the preprint of the first author \cite{Blumstein}.

As far as the authors know, the first study of degree in algebraic topology was done by Maiorana \cite{ma}, \cite{ma2} (Maiorana called the invariant ``$c(M)$", rather than ``degree"; here, we follow  the nomenclature of \cite{Bensonea}).   In particular, the paper \cite{ma2} contains some interesting calculations of degree in equivariant cohomology, presenting applications of the degree in algebraic topology which, again as far as the authors know, haven't yet been fully explored, and will be the subject of future research.

Equivariant cohomology rings are complicated graded rings, in general.  In the case where $X = \{pt \}$ is a one point space, and $G$ is a finite group, the equivariant cohomology is the cohomology ring of the group $G$, which may be defined in a purely algebraic way, and many texts and journal papers have been written on computational examples.  Quillen's theorems in \cite{Quillen1},\cite{Quillen2} tells us that the complexity of this cohomology ring (recall coefficients are taken in a field $k$ of characteristic $p$)  is directly linked to the poset of elementary abelian $p$-subgroups  of $G$, and much has been made of this since then.  On the other hand, if the group $G$ is very nice; for example, if $G$ is a torus or an elementary abelian $p$-group, but $X$ is not a point, Quillen's theorems tells us that the complexity of the equivariant cohomology ring is linked to the cohomology  structure of fixed-point sets $X^B$ of the  subgroups $B$ of $G$. Again, much has been made of this by many different authors.

The degree formula proven in this paper (without explaining now the meaning of the various terms, which come directly from Quillen's work)

\begin{equation} \deg(H^*_G(X)) = \sum_{[A,c] \in \mathcal{Q'}_{max}(G,X)}\frac{1}{|W_G(A,c)|} \deg(H^*_{C_G(A,c)}(c)) \end{equation}
really falls in line with Quillen's type of results, presenting a way to ``compute" the degree from subgroups of the group $G$ associated to elementary $p$-abelian subgroups of $G$ and fixed point sets of these elementary abelian subgroups, reducing the calculation of degree for an arbitrary group $G$ to calculations to hopefully simpler types of groups acting on spaces in a simpler way:  central extensions of an elementary abelian group $A$, with $A$ acting trivially on a connected space.  This paper does not actually present the study of any sample computations for these sorts of actions;  we hope to return to this study in future work.  However, see the paper of Lynn \cite{Lynn} for an example when $X$ is a point.

%This additivity formula resembles many of the localization-type results common in equivariant cohomology. 

Returning to our introduction of the ideas presented in this paper, given any compact Lie group $G$, one may smoothly embed it into a unitary group $U$. For a given prime $p$, define the space $S$ to be the space of all diagonal matrices of order $p$ in $U$. The resulting quotient $F \doteq U/S$ is a smooth $G$-manifold (known in some circles as Quillen's magic space!). Lynn derives her additivity formula, which is the additivity formula (1) above, in the case $X = \{pt\}$,  by cleverly using geometric properties of the orbits of $G$ on certain sub-manifolds of $F$. These orbits carve up the space in just the right way (Theorems 5.1 and 7.1 of \cite{Lynn}) so that on the algebraic side, the degree invariant adds over the components. For this reason, we refer to Lynn's additivity formula as ``geometric".

On the other hand, if one appeals only to commutative algebra, there is a different additivity formula for the degree of a graded $S$-module $M$, which we will call the ``algebraic" additivity formula. Specifically, 
\begin{equation} \deg(M) = \sum_{\p \in \mathcal{D}(M)}\ell_{S_{[\p]}}(M_{[\p]})\cdot \deg(S/\p), \end{equation} 
where $\mathcal{D}(M)$ is the set of minimal primes $\p$ such that $\dim_S(S/\p)=\dim_S(M)$. Here, $\dim_A(X)$ is the Krull dimension of the $A$-module $X$, $\ell_A(X)$ is the length, in the algebraic sense, of the $A$-module $X$, and $Y_{[\p]}$  is the graded localization of $Y$ at a graded prime ideal $\p$.

Lynn concludes her paper asking how to relate this ``algebraic" formula (2) to her version of the ``geometric" formula (1).  
The main theorem (\ref{theorem main thm geometric degree})--additivity formula (1)--of the present paper makes this connection precise, while also significantly generalizing Lynn's result. This theorem is proven for a compact Lie group $G$ acting on a topological space $X$ which is either compact or has finite mod $p$ cohomological dimension.

While Quillen's theorems say that the sums on the right side of (1) and (2) have essentially the same index sets, we show that the geometric additivity formula (1) is term-by-term equal with the algebraic additivity formula (2) stated above. In fact, the proof of our additivity formula (1) uses the algebraic additivity formula (2) twice.  Thus, the connection between commutative algebra and geometry is built into our methodology. 

By contrast, the key tool in Lynn's proof of the special case of (1), where $X$ is a point,  is the Gysin sequence of differential topology, which depends fundamentally on the existence of a Riemannian metric on a smooth manifold. Thus, Lynn's results leading to her proof of the special case of (1) are stated only in the category of smooth $G$-manifolds.  Because our methods of proof rely mainly on commutative algebra, we were able to get away from some of the more complicated geometric/topological arguments Lynn resorts to. Consequently, the hypotheses of our results are far less restrictive (no smoothness assumptions) and agree with the hypotheses of Quillen's theorems.

\section{Summary of Quillen's results}
Most of the results of this section can be found in Quillen \cite{Quillen1}.  For any compact Lie group $G$, there exists a ``universal principal bundle" for $G$, denoted $EG \rightarrow BG$. (See Milnor for a construction \cite{Milnor}.)  Recall that $G$ acts freely on $EG$ and $EG/G \cong BG$. 

If $X$ is a topological space on which $G$ acts continuously (a ``$G$-space"), $G$ acts diagonally on $EG \times X$, and one can consider the associated fiber bundle: $EG \times _G X \doteq (EG \times X) / G$. Then, the \textit{equivariant cohomology ring} is defined by $$H^*_G(X) \doteq H^*(EG \times_G X).$$  The coefficient ring isn't noted always, but generally we'll consider coefficients in a field whose characteristic divides the order of $G$, if $G$ is finite.

For our purposes, $H^*$ represents singular cohomology, and the product is cup product. In Quillen's formulation \cite{Quillen1} he uses sheaf cohomology, but we don't  lose any of the fundamental properties by switching to singular cohomology (this is explained on pg. 1163 in \cite{Symonds}.) 
 
If $G$ and $G'$ are compact Lie groups with $X$ a $G$-space and $X'$ a $G'$-space, consider a pair of maps $(u,f): (G,X) \rightarrow (G',X')$ such that $u$ is a Lie group homomorphism, and $f$ a continuous map that is $u$-equivariant, i.e. $f(gx)= u(g)f(x)$ for all $g \in G$ and all $x \in X$. These are the morphisms in the category of pairs $(X,G)$, with $X$ a $G$-space, that we are considering.  With respect to such morphisms, equivariant cohomology is contravariantly functorial \cite{Quillen1}.  In particular, if $Y \subseteq X$, $H$ is a subgroup of $G$ and $u$ and $f$ are inclusion mappings, then the induced map $H^*_G(X) \rightarrow H^*_H(Y)$  is called the restriction map and usually denoted $res_H^G$.  Also, if $h \in G$, consider the automorphism $(c_h, \mu_h): (G,X) \rightarrow (G,X)$, where $c_h$ is defined for each $g \in G$ by $c_h(g) \doteq hgh^{-1}$ and $\mu_h$ is defined for each $x \in X$ by $\mu_h(x) = hx$.  We will call the map on equivariant cohomology induces by such a map an ``inner automorphism".  In \cite{Quillen1}, one finds the proof of
\begin{lemma} \label{lemma inner autos act trivially}
Inner automorphisms act trivially on equivariant cohomology: If $X$ is a $G$-space, then for every $h \in G$, $(c_h,\mu_h)^*:H^*_G(X) \rightarrow H^*_G(X)$ is the identity map.  
\end{lemma}

\subsection{Commutative Algebra of Equivariant Cohomology Rings}
Quillen laid the foundation for the study of the commutative algebra of equivariant cohomology rings.  The  theorems of Quillen's papers (\cite{Quillen1}, \cite{Quillen2}) that we need here are presented below.  The two main theorems relate the prime spectrum of $H^*_G(X)$ to the elementary abelian groups which appear as subgroups of $G$, and have fixed points when acting on $X$.

Standard properties of cup product say that, in general, the graded ring $H^*_G(X) \doteq \oplus_{i \geq 0} H^i_G(X)$ will not be a commutative ring--although when the characteristic of the coefficient ring is 2, we do get commutativity of the product.    In order to directly apply results from commutative algebra, we make the following definition for any commutative coefficient ring $k$:     If $char(k)$ is odd or $0$, $H_G(X,k) \doteq \oplus_{i \geq 0} H^{2i}_G(X,k) $ is the even degree part of $H^*_G(X)$. If $char(k) =2$ then $H_G(X,k) \doteq H^*_G(X,k)$.   Of course, $H_G(X)$ is always a commutative graded ring.

The following results allow for the study of geometric and commutative algebraic properties of equivariant cohomology rings; coefficients will always be taken in a field $k$.

\begin{theorem}(\cite{Quillen1} Corollary 2.2) \label{theorem Quillen finiteness 1}
Let $G$ be a compact Lie group, $X$ be a $G$-space and suppose that $H^*(X)$ is a finitely generated $k$-vector space.   Then, $H^*_G(X)$ is a finitely generated graded $k$-algebra.
\end{theorem}

It's not hard to see that this implies 
\begin{corollary} With the hypotheses of Theorem 2.2, $H_G(X)$ is a finitely generated commutative graded $k$-algebra.  In particular, $H^0_G(X)$ is a finite dimensional algebra over $k$ and is thus an Artinian ring.
\end{corollary}

Another ``finite generation" result we shall use is

\begin{theorem}(\cite{Quillen1} Corollary 2.3)\label{theorem Quillen finiteness 2}  
If $(u,f): (G,X) \rightarrow (G',X')$ is a morphism such that $u$ is injective and $H^*(X)$ is a finitely generated $k$-module, then  the ring homomorphism $(u,f)^*:H^*_{G'}(X') \rightarrow H^*_G(X)$ is finite; i.e. with respect to this homomorphism, $H^*_G(X)$ is a finitely generated $H^*_{G'}(X')$-module. 

Also, $H^*_G(X)$ and $H_G(X)$ are  finitely generated graded modules over the commutative ring $H_{G'}(X')$, with respect to the same ring homomorphism $(u,f)^*$, restricted to $H_{G'}(X')$.
\end{theorem}

When relating the equivariant cohomology ring of a subgroup $G$ to that of a larger group $G'$, the theorem above allows us to take advantage of the theory of \textit{integral extensions} from commutative algebra.  A more thorough explanation of the theory of integral extensions can be found in \cite{Eis}. For a subring $R$ of $S$, an element $s \in S$ is said to be \textit{integral} over $R$ if it is the root of a monic polynomial with coefficients in $R$.  We say that $S$ is integral over $R$ if every element of $S$ is integral over $R$.  When $f:R \rightarrow S$ is a ring homomorphism, we say that $S$ is integral over $R$ (with respect to $f$) when $S$ is integral over its subring $f(R)$. 

Given a ring homomorphism $f:R \rightarrow S$ where $S$ is a finitely generated $R$-module with respect to $f$, a standard result from commutative algebra says that $S$ must also be integral over $R$. In the case of equivariant cohomology rings, if $(u,f): (G,X) \rightarrow (G',X')$ is a morphism such that $u$ is injective and $H^*(X)$ is a finitely generated $k$-module (where $k$ is the coefficient field), we see  that $H_G(X)$ is integral over $H_{G'}(X')$ (with respect to $(u,f)^*$).  For this paper, the fact we need about integral extensions is:

\begin{itemize}
\item ``Lying over": For an integral extension $f: R \rightarrow S$, for every $\p \in Spec(R)$  containing $\ker f$, there exists a $\q \in Spec(S)$ such that $f^{-1}(\q) = \p$.

\end{itemize}

\subsection{Main Theorems of Quillen}
Quillen's main theorems hold in the following setting, which we assume for the remainder of the paper: $G$ is a compact Lie group which acts continuously on a space $X$.  We also require that $X$ is Hausdorff, and that $X$ is either compact, or is paracompact with finite mod-$p$ cohomological dimension (see \cite{Quillen1} for a definition.) Finally, all cohomology is taken with coefficients in the prime field $k \doteq \mathbb{F}_p$ for $p$ a prime number, which is fixed throughout the following discussion, and $H^*(X)$ is a finite dimensional $k$-vector space. 

Recall that an \textit{elementary abelian} $p$-subgroup of $G$ is a subgroup $A$ of $G$ such that $A$ is isomorphic to a direct product of a finite number $r$ of cyclic groups $\mathbb{Z}/p\mathbb{Z}$.  The number of factors $r$ is called the rank of $A$.   Since $p$ will be taken as a fixed prime number throughout we often omit the reference to $p$.

Given a pair $(G,X)$ which satisfy the conditions above, we denote \textit{Quillen's category of pairs} by $\mathcal{Q}(G,X)$.  Objects of this category are pairs $(A,c)$ where $A$ is an elementary abelian $p$-subgroup of $G$, $X^A \neq \emptyset$, and $c$ is a connected component of $X^A$. If $(A,c)$ and $(A',c')$ are objects in $\mathcal{Q}(G,X)$, then there is a morphism between them if there exists an element $g\in G$ such that $gAg^{-1} \leq A'$ and $c' \subseteq gc$.

If there exists a morphism $(A,c) \rightarrow (A',c')$ in $\mathcal{Q}(G,X)$, we say that $(A,c)$ is \textit{subconjugate} to $(A',c')$ or $(A,c) \lesssim (A',c')$.  We define \textit{conjugate} objects, denoted $(A,c) \sim (A',c')$, to be two objects that are isomorphic in $\mathcal{Q}(G,X)$; isomorphism is an equivalence relation on $\mathcal{Q}(G,X)$ and the equivalence class of $(A,c)$ is denoted by $[A,c]$.

By definition, an element $(A,c) \in \mathcal{Q}(G,X)$ is a \textit{maximal} element of $\mathcal{Q}(G,X)$ if and only if any morphism with source $(A,c)$ is an isomorphism.  As Quillen \cite{Quillen2} notes, $(A,c)$ is a maximal element of $\mathcal{Q}(G,X)$ if and only if $A$ is a maximal elementary abelian subgroup of $G_x$ for every $x \in c$.

\begin{definition}
For a $G$-space $X$, $$\mathcal{Q'}(G,X) \doteq \{ [A,c] : (A,c) \in \mathcal{Q}(G,X), \;  (A,c) \;\mbox{is a maximal element of} \; \mathcal{Q}(G,X) \}.$$  Within $\mathcal{Q'}(G,X)$, we are often interested in objects corresponding to  elementary abelian subgroups whose rank is largest, so we define $$\mathcal{Q'}_{max}(G,X) \doteq \{ [A,c]  \in \mathcal{Q'}(G,X) : rk(A) \geq rk(B), \; \forall \; [B,d] \in \mathcal{Q'}(G,X) \}.$$
\end{definition}

We refer the reader to \cite{MilgramPriddy} for an example where $\mathcal{Q'}_{max}(G,X)$ is distinct from $\mathcal{Q'}(G,X)$; in this case, $G=GL_n(\mathbb{Z}/p)$ for $n \geq 4$.

\begin{definition}
Let $(A,c) \in \mathcal{Q}(G,X)$, and pick a point $x_0 \in c$.  We define $\p_{A,c}$ as the kernel of the following composition
$$H_G(X) \xra{res_A^G} H_A(x_0) \rightarrow H_A(x_0)/ \sqrt{0}.$$
\end{definition}

Since $H_A/\sqrt{0}$ is a polynomial ring, using the well-known computation of $H^*_A$, we see that $\p_{A,c} \in Spec(H_G(X))$. 

\begin{theorem} (\cite{Quillen2} Proposition 11.2) \label{theorem Quillen main 2}   For a $G$-space $X$, 
\begin{itemize}
\item  $\p_{(A,c)}\supseteq \p_{(A',c')}$ if and only if $(A,c) \lesssim (A',c').$  Also, $\p_{(A,c)} = \p_{(A',c')} $ if and only if $(A,c) \sim (A',c').$
\item The correspondence $[A,c] \mapsto \p_{(A,c)}$ defines a bijective function between the elements of $\mathcal{Q'}(G,X)$ and the set of minimal prime ideals in $H_G(X).$\end{itemize}
\end{theorem}

As a corollary,
\begin{theorem} (\cite{Quillen1})  \label{theorem Quillen main 1}  For a $G$-space $X$, the Krull dimension of $H_G(X)$ equals the maximal rank of an elementary abelian $p$-subgroup such that $X^A \neq \emptyset$. 
\end{theorem}

For example, in the case where $X$ is a point, Quillen's theorem says that Krull dimension of $H_G$ equals the maximal rank of an elementary abelian $p$-group in $G$.

\subsection{Quillen's Magic Space $F$}

For any compact Lie group $G$, there exists a (not unique) unitary group $U$ and a closed embedding of $G$ into $U$. Given a compact Lie group $G$, we often begin by fixing such an embedding. Then, define $S$ to be the set of diagonal matrices of order $p$ in $U$, where $p$ is a fixed prime.  Observe that $U$ is equipped with both a left $G$-action and a right $S$-action.  

The topological space  $F\doteq U/S$ then has a natural $G$-action: if $g \in G$ and $uS \in F$, then the action of $G$ on $F$ is defined by $g \cdot(uS) \doteq guS$. Quillen showed how the product space $X \times F$ can be used to encode information about how $G$ acts on $X$. In fact, Quillen used the following result to prove the two main theorems of the last section(\ref{theorem Quillen main 1}, \ref{theorem Quillen main 2}), and we make use of this result in the next section.  Recall that a sequence of modules  $\xymatrix{ A \ar[r]^h &B \ar@<2pt>[r]^f \ar@<-2pt>[r]_g &C}$ is called an \textit{equalizer sequence} when $0 \xra{} A \xra{h} B \xra{f-g} C$ is an exact sequence.

\begin{lemma}(\cite{Quillen1}, Lemma 6.5 and preceding discussion) \label{lemma Q formula}
The following is an equalizer sequence of $H_G(X)$-modules:
$$H^*_G(X) \rightarrow H^*_G(X \times F) \rightrightarrows H^*_G(X \times F \times F);$$
defined by applying the equivariant cohomology functor to the sequence \begin{equation*}
\xymatrix{
X \times F \times F  \ar@<-.5mm>[r]_{\pi_{13}} \ar@<.5mm>[r]^{\pi_{12}} & X \times F \ar[r]^{\pi_1} & X,  \\
}
\end{equation*} where $\pi$ is the projection map onto the indicated components. 

Quillen then shows there is an isomorphism of $H_G(X)$-modules
$$H^*_G(X \times F) \cong H^*_G(X) \otimes_{k} H^*(F).$$  Here, if $r \in H_G(X)$, $x \otimes y \in H^*_G(X) \otimes_{k} H^*(F)$, then the $H_G(X)$-module structure is defined by $r(x \otimes y) \doteq (rx) \otimes y$.
\end{lemma}

\section{Summary of Duflot's Localization Results}
The material developed in this section is leveraged to prove the main theorem of this paper (Theorem \ref{theorem main thm geometric degree}). We do not present the details of proofs here, but we do give some discussion. 
\begin{definition}
Let $G$ be a group, $X$ a $G$-space, and let $(A,c) \in \mathcal{Q}(G,X)$.
\begin{itemize}
\item $N_G(A,c) \doteq \{g \in G: gAg^{-1}=A \} \cap \{g \in G : gc =c \}$
\item $C_G(A,c) \doteq \{g \in G : ga = ag, \; \forall a \in A\} \cap \{g \in G : gc=c\}$
\item $W_G(A,c) \doteq N_G(A,c)/C_G(A,c)$
\end{itemize}
In the case that $X$ is just a point, we omit the $c$'s from the notation. e.g. $N_G(A)$ instead of $N_G(A,c)$.
\end{definition}

If $Z$ is any $G$-space and $(A,c) \in \mathcal{Q}(G,Z)$, then $N_G(A,c)$ acts on $Z^A$, since $N_G(A,c) \subseteq N_G(A)$; of course, $N_G(A,c)$ acts on $c$ as well.  We will also need the following notation:  if $Y \subseteq Z$, then $G\cdot Y \doteq \{ g\cdot y \mid y \in Y\}$ is the indicated $G$-subspace of $Z$.

Recall some definitions from graded commutative algebra:  if $L$ is a $\mathbb{Z}$-graded commutative ring, $M$ is a graded $A$-module and $S$ is a multiplicatively closed subset consisting entirely of homogeneous elements of $L$, then the abelian group $S^{-1}M$ has a grading in the usual way:  the homogenous elements of $S^{-1}M$ of degree $j$ are the elements $m/s$ where $m \in M$ is homogeneous and $j = deg(m)-deg(s)$.  With this definition, $S^{-1}L$ is a graded ring, and $S^{-1}M$ is a graded $S^{-1}L$-module.  
\begin{definition}  Let $\p$ be a graded prime ideal in $L$, and $M$ a graded $L$-module.
\begin{itemize}
\item $M_{[\p]}$ is the localization $S^{-1}M$ where $S$ is the set of homogeneous elements of $L$ not in $\p$.  We will say that $M_{[\p]}$ is the \textit{graded localization} of $M$ at $\p$.
\item $M_{(\p)}$ is the ungraded (or ``concentrated in degree zero") module consisting of the elements of $M_{[\p]}$ of degree zero.
\item $M_{\p}$ is the ungraded module $\tilde{S}^{-1}M$ where $\tilde{S}$ is the set of all elements of $L$ not in $\p$.
\end{itemize}
\end{definition}

We won't use the second notion of localization in the list above, which is that used in algebraic geometry \cite{GRO}, in this paper.  We'll say that the third type of localization above is ``ordinary" localization.

The main result of \cite{Duflot} is: 
\begin{theorem} (\cite{Duflot} Theorem 3.2)
\label{theorem Duflot localization}
Suppose that $[A,c] \in \mathcal{Q'}(G,X)$. Then there is an isomorphism of $H_G(X)_{[\p_{(A,c)}]}$-modules $$H^*_G(X)_{[\p_{(A,c)}]} \xra{(res^G_{C_G(A,c)})_{[\p_{(A,c)}]}} H^*_{C_G(A,c)}(c)^{W_G(A,c)}_{[\p_{(A,c)}]}.$$
\end{theorem}

\begin{remark}In the original paper \cite{Duflot}, it was not clear whether the localization was ordinary localization or graded localization. In \cite{Duflot3}, Duflot clarifies this, and we've added the brackets here to indicate graded localization.   \end{remark}
 
Let's discuss this result a bit, giving an outline of the proof.  Our outline focuses on results we will need in the final section of the paper, and we do not give proofs. Suppose that an embedding of $G$ into a unitary group $U$ is fixed, and $F$ is the $G$-space $U/S$ as before.  

First,  $N_G(A,c)$ acts on $c$ and hence on $c \times F^i$ and on $c \times (F^A)^i$ diagonally.   Furthermore, these actions induce actions of  $W_G(A,c)$ on $H^*_{C_G(A,c)}(Z)$, whether $Z$ is $c, c \times F^i$ or $c \times (F^A)^i$:    for any $n \in N_G(A,c)$, we have the equivariant conjugation map on the pair $(C_G(A,c),Z) \rightarrow (C_G(A,c),Z)$ which is the pair $g \mapsto ngn^{-1}$ and $z \mapsto n\cdot z$.  Refer to this map as $ n$, and the induced ring automorphism on equivariant cohomology as $n^*: H^*_{C_G(A,c)}(Z) \rightarrow H^*_{C_G(A,c)}(Z)$.  Recall that inner automorphisms act trivially on equivariant cohomology, therefore if $n \in C_G(A,c)$ then $n^*=id$ on $H^*_{C_G(A,c)}(Z)$. Therefore, $W_G(A,c)$ has a well-defined action on $H^*_{C_G(A,c)}(Z)$. 

Now, for  the spaces $Z$ in the paragraph above, there are equivariant maps $\theta:(C_G(A,c), Z) \rightarrow (G,X)$:   for the groups involved, we use inclusions, and for the spaces, if $Z = c$,  we use inclusion $c \rightarrow X$, while  for $Z = c \times F^i$ or $Z = c \times (F^A)^i$, $\theta$ is the inclusion $Z \rightarrow X \times -$, followed by projection onto the first factor. We point out that all of the induced maps on equivariant cohomology  $\theta^*:H^*_G(X) \rightarrow H^*_{C_G(A,c)}(Z)$ have the property that, for every (homogeneous) $x \in H^*_G(X), y \in H^*_{C_G(A,c)}(Z)$ and every $w \in W_G(A,c)$, $\theta^*(x)(w \cdot y) = w \cdot (\theta^*(x)y)$; in other words, the automorphisms induced by the elements of $W_G(A,c)$  on $H^*_{C_G(A,c)}(Z)$ are $H_G(X)$-module maps.  Note that this implies the image of $\theta^*$ lands in the invariant subring $H^*_{C_G(A,c)}(Z)^W$ since every $w \in W$ induces a ring automorphism preserving the multiplicative identity $1$ of the ring: $$ \theta^*(x) = \theta^*(x)(w \cdot 1) = w \cdot \theta^*(x).$$

To see this, look at the commutative diagrams below; $n$ is always the map $g \mapsto ngn^{-1}$, $z \mapsto nz$, for every $n \in N_G(A,c)$:
 
 \begin{equation*}
\xymatrix{
(G,X) \ar[r]^{ n}  & (G,X)\\
(C_G(a,c),Z) \ar[u]^{\theta} \ar[r]^{ n} & (C_G(A,c),Z) \ar[u]_{\theta}\\
}
\end{equation*}
Thus there is a commutative diagram
\begin{equation*}
\xymatrix{
H^*_G(X) \ar[d]_{\theta^*} & H^*_G(X) \ar[l]_{n^*} \ar[d]_{\theta^*} \\
H^*_{C_G(A,c)}(Z)  & H^*_{C_G(A,c)}(Z)  \ar[l]_{n^*} \\
}
\end{equation*}
Now, lemma \ref{lemma inner autos act trivially} implies that the upper ring homomorphism $n^*$ on $H^*_G(X)$ is the identity map, so commutativity implies that 
$$n^*(\theta^*(x)y) = n^*(\theta^*(x))n^*(y) = \theta^*(n^*(x))n^*(y) = \theta^*(x)n^*(y),$$ which is our point:  by definition, if $w$ is represented by $n \in N_G(A,c)$, $w\cdot ? \doteq n^*(?)$.

Duflot passes a problem about $H^*_G(X)$ to the space $H^*_G(X \times F)$ using Quillen's equalizer sequence  \ref{lemma Q formula}; first, the following two results are proved:
\begin{lemma}\label{lemma W acts freely} (\cite{Duflot} pg. 98, 99) Suppose that $X$ is a $G$-space and  $(A,c) \in \mathcal{Q}(G,X)$; note that the subgroup $N_G(A,c)$ of $G$ acts  on $c \times F^A$ (diagonally).
\begin{itemize}
\item[a)] The set of components of $c \times (F^A)^i$, $\pi_0(c \times (F^A)^i)$,  is a finite set.   The action of $N_G(A,c)$ on $c \times (F^A)^i$ induces an action of $N_G(A,c)$ on $\pi_0(c \times (F^A)^i)$ such that  if $d$ is a component of $c \times (F^A)^i$, then for every $g \in C_G(A,c)$, $g \cdot d = d.$  Therefore, $W_G(A,c)$ acts on the set of components of $c \times (F^A)^i.$
\item[b)] 
$W_G(A,c)$ acts freely on the set of components of $c \times (F^A)^i, i=1,2$.
\end{itemize}
\end{lemma}

Although the corollary below is stated in \cite{Duflot}, we'll provide details about why it follows from \ref{lemma W acts freely} later in this paper, in \ref{theorem length of W action}.
\begin{corollary} (\cite{Duflot} Lemma 3.5) \label{corollary W acts freely}
If $(A,c) \in \mathcal{Q}(G,X)$, then $H^q_{C_G(A,c)}(c \times (F^A)^i)$ is a free $k[W_G(A,c)]$-module for all $i \geq 1$ and all $q \geq 0$. 
\end{corollary}

Now, Duflot's proof of her localization theorem uses results about the following sequence of homomorphisms of graded rings,  for a pair $[A,c] \in \mathcal{Q'}(G,X)$:

\begin{equation}
H^*_G(X) \xra{\textcircled{1}} H^*_G(G\cdot(c \times F^A)) \xra{\textcircled{2}} H^*_{N_G(A,c)}(c \times F^A) \xra{\textcircled{3}} H^*_{C_G(A,c)}(c \times F^A)^{W_G(A,c)} \tag{$\star$}
\end{equation}
 The sequence of facts from \cite{Duflot} that we will need concerning these maps are as follows.

\begin{lemma}(\cite{Duflot} Lemma 3.4)
For  $[A,c] \in \mathcal{Q'}(G,X)$, there is an isomorphism (the map \textcircled{2} in ($\star$) above)
$$H^*_G(G\cdot(c \times F^A)) \cong H^*_{N_G(A,c)}(c \times F^A)$$
\end{lemma}
\begin{lemma} (\cite{Duflot} Lemma 3.6)
For $(A,c) \in \mathcal{Q}(G,X)$, there is an isomorphism (the map \textcircled{3} in ($\star$) above) $$H^*_{N_G(A,c)}(c \times F^A) \cong H^*_{C_G(A,c)}(c \times F^A)^{W_G(A,c)}.$$
\end{lemma}  
Finally, one takes graded localizations at minimal primes and obtains

\begin{theorem} (\cite{Duflot} pg. 100) \label{theorem Duflot diagram}
Fix an $[A,c] \in \mathcal{Q'}(G,X)$ with corresponding minimal prime $\p_{(A,c)} \doteq \p$.  The following diagram is commutative, each of the vertical arrows are graded isomorphisms of $R_{[\p]}=H_G(X)_{[\p]}$-modules, and the first and last rows are equalizer sequences.
\begin{equation*}
\xymatrix{
   H^*_G(X)_{[\p]} \ar[r] & H^*_G(X \times F)_{[\p]} \ar[d]^{\textcircled{1}} \ar@<2pt>[r] \ar@<-2pt>[r]  & H^*_G(X \times F^2)_{[\p]} \ar[d] \\
   &H^*_G(G \cdot (c \times F^A))_{[\p]}  \ar[d]^{\textcircled{2}} \ar@<2pt>[r] \ar@<-2pt>[r]  & H^*_G( G \cdot (c \times (F^A)^2))_{[\p]}  \ar[d] \\   
   &H^*_{N_G(A,c)}(c \times F^A)_{[\p]}  \ar[d]^{\textcircled{3}} \ar@<2pt>[r] \ar@<-2pt>[r]  & H^*_{N_G(A,c)}( G \cdot (c \times (F^A)^2))_{[\p]} \ar[d] \\
   &H^*_{C_G(A,c)}(c \times F^A)_{[\p]}^{W_G(A,c)} \ar@<2pt>[r]  \ar@<-2pt>[r]  & H^*_{C_G(A,c)}( G \cdot (c \times (F^A)^2))_{[\p]}^{W_G(A,c)} \\
H^*_{C_G(A,c)}(c)_{[\p]}^{W_G(A,c)} \ar[r] & H^*_{C_G(A,c)}(c \times F)_{[\p]}^{W_G(A,c)} \ar[u]_{\textcircled{4}} \ar@<2pt>[r] \ar@<-2pt>[r]  & H^*_{C_G(A,c)}( G \cdot (c \times F^2))_{[\p]}^{W_G(A,c)} \ar[u]\\
}
\end{equation*}
\end{theorem}

Note that, before localizing, all modules in the diagram are $R \doteq H_G(X)$-modules, using restriction maps.  Taking the claims of this theorem for granted, Duflot's localization result is obtained from this diagram since each square is commutative and all of the vertical arrows are isomorphisms.  The crucial vertical arrow $\textcircled{1}$ comes from a restriction map, but is only an isomorphism after graded (or ungraded) localization at $\p \subseteq H_G(X)$; the details are explained in lemma 3.3 of \cite{Duflot} (which we have not included here).  Similarly,  the bijectivity of the arrow $\textcircled{4}$, which comes from a restriction map,  is also an application of lemma 3.3 of \cite{Duflot} in a different case.

The exactness of the first and final rows of Duflot's diagram \ref{theorem Duflot diagram} come from lemma \ref{lemma Q formula}. Specifically, apply the result to the pairs $(G,X)$ and $(C_G(A,c), c)$ respectively, to get the equalizer sequences: $H^*_G(X) \rightarrow H^*_G(X \times F) \rightrightarrows H^*_G(X \times F^2)$, and $H^*_{C_G(A,c)}(c) \rightarrow H^*_{C_G(A,c)}(c \times F) \rightrightarrows H^*_{C_G(A,c)}(c \times F^2)$.  Since localization is exact, and taking invariants is left exact, we get exactness of the top and bottom rows of Duflot's diagram \ref{theorem Duflot diagram}, and conclude that $H^*_G(X)_{[\p]} \cong H^*_{C_G(A,c)}(c)_{[\p]}^{W_G(A,c)}$.

\section{The Degree of an Equivariant Cohomology Ring}
In this section we state the definition of the degree of a graded module, and state relevant results. This means that there will be two different usages of ``degree":  the degree of a homogeneous element in a graded ring, and the degree of the ring itself.  The usages are clearly distinguished by context.

 Suppose $L$ is a  $\mathbb{Z}$-graded Noetherian ring;  we denote the category of finitely-generated graded $L$-modules  by $\grmod(L)$. 
 
 In this paper, all of our graded rings will be of one of two types:
 \begin{itemize}
 \item[I.] A positively graded Noetherian ring $R$ ($R_i  = 0$ for $i <0$), with $R_0$ an Artinian ring that is a finitely generated $k$-algebra:  thus $R_0$ must be finite dimensional as a vector space over $k$.
 \item[II.] A graded localization of a positively graded ring of  type I at a graded prime ideal.
 \end{itemize}

We use an asterisk to denote a usual definition in non-graded algebra,  adapted for the graded category; this is fairly standard notation--e.g., it is used in \cite{brhe}. For example, if $M \in \grmod(L)$ is a graded $L$-module, we denote its graded Krull dimension by $*dim_L(M)$, i.e. this is the longest chain of graded prime ideals in $L$ containing the $L$-annihilator of $M$. For the type of rings that interest us, it turns out that  graded and ungraded measures  are directly comparable; for some measures this is well-known, in other cases we will refer to \cite{Blumstein} for discussion and proofs. We choose to use the *-notation in any case, as it is often practical to work only with homogeneous elements.  

From now on, we will use $R$  to denote a general ring of type I.   

For $M \in \grmod(R)$, the Poincar\'{e} series of $M$ may be defined by $P_M(t)=\sum_i \vdim_k(M_i)t^i$, where $\vdim_k(V)$ is the vector space dimension of the vector space $V$ over $k$.  To see this, if  $M \in \grmod(R)$, then, for every $j$,  $M_j$ is an Artinian $R_0$-module (since $M_j$ is a finitely generated $R_0$-module), and  $ \ell_{R_0}(M_j) = \vdim_k(M_j)$.  This follows since, for every maximal ideal $\m_0$ of $R_0$, $R_0/\m_0$ is isomorphic to the field $k$.  Therefore,
$$P_M(t) =\sum_i \vdim_k(M_i)t^i =\sum_i \ell_{R_0}(M_i)t^i.$$

Recall that, since $R$ is of type I,  $*dim_R(M) = dim_R(M)$ and $*dim_R(M)$ is the order of the pole of $P_M(t)$ at $t=1$ (\cite{Blumstein} contains an account of this, but these facts are well-known and hold for any positively graded Noetherian ring whose elements in degree zero form an Artinian ring).

\begin{definition}  If $R$ is of type I, $M\in \grmod(R), \newline M \neq 0$ and $D(M) = \sdim_R(M)$, then
$$\deg_R(M) \doteq \lim_{t \rightarrow 1} (1-t)^{D(M)}P_M(t)$$ is a well-defined, strictly positive, rational number.  For convenience, define $\deg_R(0) = 0.$
\end{definition}

Given $M \in \grmod(R)$,  $M \neq 0$,
we can read off the degree of a module directly from the Poincar\'{e} series if we expand it as a Laurent series about $t=1$: $$P_R(t) = \frac{\deg(M)}{(1-t)^{D(M)}} + \text{``higher order terms"}.$$

The degree of $M$ can be regarded as a species of non-integer multiplicity associated to $M$.  As noted in the Introduction, this rational number is called $c(M)$ in \cite{ma}; as far as we know, this is the first reference that discusses this numerical invariant for equivariant cohomology.  However, we use the terminology ``degree", as used in \cite{Bensonea}.  The reference \cite{Vasconcelos} uses ``degree" to refer to the integer invariant of local algebra also called ``multiplicity".  The first author's preprint \cite{Blumstein} exposits the concept of a ``Samuel multiplicity"  associated to certain ideals in a graded ring, and makes precise the relationship between the degree and this Samuel multiplicity.  We don't make this more precise here, because making all the definitions necessary to state the relationship would add to the length of this paper and is not used anywhere here; instead we refer the reader to \cite{Blumstein}.

We may make the following definitions:
\begin{definition} Let $L$ be a $\mathbb{Z}$-graded ring.
\begin{itemize}

\item  A nonzero graded $L$-module $M$ is  a *simple $L$-module if and only if the only graded $L$-submodules of $M$ are $0$ and $M$. 
\item A *composition series for a graded $L$-module $M$ is a sequence of graded submodules of $M$ of the form 
$$0 = M_0 \subset M_1 \subset \cdots \subset M_n = M$$ such that for each $i$, $M_{i+1}/M_i$ is a *simple $L$-module.  The *length of this *composition series is defined to be $n$.
\end{itemize}
\end{definition} 

The expected properties of *composition series hold and we may talk about the *length of an $L$-module $M$ (denoted $*\ell_L(M)$) . One can see that $*\ell_L$ adds over short exact sequences of graded $L$-modules; another straightforward lemma is:
\begin{lemma} \label{length of a tensor product}
Let $L$ be a $\Z$-graded  ring with  $L_0$ is a finite dimensional vector space over a field $k$. (This implies that $L_0$ is an Artinian ring.)  Suppose that  $M$ is a graded $L$-module and $V$ is a graded finite dimensional vector space over $k$, consider $M \otimes_k V$ as a graded $L$-module defined by $r \cdot (m \otimes v) \doteq (r \cdot m) \otimes v$ for $r \in L$ (with the usual grading on the tensor product).  If $*\ell_L(M)$ is finite, so is $*\ell_L(M  \otimes_k V)$ and $*\ell_L(M \otimes_k V) = *\ell_L(M) t$, where $t$ is the (total) dimension of $V$ as a vector space over $k$. 
\end{lemma}

For comparisons to ungraded concepts and further discussion, we refer to \cite{Blumstein}; to give some context to results proved here, a sample fact from \cite{Blumstein} is
\begin{lemma}\label{finite length lemma}  (Theorem 4.10 of \cite{Blumstein})  Let $R$ be of type I, and let $\p$ be a minimal prime for $M \in \grmod(R)$.  Then, $\p$ is graded and $*\ell_{R_{[\p]}}(M_{[\p]}) = \ell_{R_{\p}}(M_{\p}) < \infty$.
\end{lemma}

The following theorem gives an expected additivity formula for degree, where the sum is indexed by the minimal primes of maximal dimension.  This theorem just states the behavior expected of some multiplicity-type numerical invariants, and its proof is similar  to  corresponding statements for other sorts of multiplicities.  For example, see \cite{se} for what Serre  calls Samuel multiplicity for local rings; \cite{Vasconcelos} gives an account of some other types of multiplicity. The main theorem of this paper (\ref{theorem main thm geometric degree}) places this ``algebraic" additivity formula in a geometric context when applied to equivariant cohomology rings. 
 
\begin{theorem} (see, e.g., \cite{Blumstein} Theorem 6.7) \label{theorem sum formula for degree and multiplicity}  Suppose $R$ is of type I.
Let $M \in \grmod(R)$,  and $\mathcal{D}(M)$ be defined as the set of prime ideals $\p$ in $R$, necessarily minimal primes for $M$ and graded, such that $\sdim_R(R/\p) = \sdim_R(M)$.  Then, $$\deg(M) = \sum_{\p \in \mathcal{D}(M)}*\ell_{R_{[\p]}}(M_{[\p]})\cdot \deg(R/\p).$$
\end{theorem}

\subsection{Degree Decomposition of Lynn}
This short section serves simply to give a snapshot of Lynn's results (\cite{Lynn}). Her main result is a geometric additivity formula  for the degree of an equivariant cohomology ring in the case where $X$ is taken to be a point.  This work was the inspiration for our more general additivity formula (Theorem \ref{theorem main thm geometric degree}); however, we do not use any of Lynn's results in this paper.

As noted in the introduction, Lynn's methods rely on the smoothness of the $G$-space $X$.  In order to prove her main theorem, she uses Gysin-type sequences to prove
\begin{theorem}(\cite{Lynn} Theorem 4.21) \label{theorem Becky 4.21}
Let $G$ be a compact Lie group, and let $X$ be a smooth, compact $G$-manifold.  Let $Z= \cup_{i=1}^n Z_i$, where the $Z_i$'s are closed, $G$-invariant, disjoint submanifolds of $X$ such that $\nu_{Z_i}$ (the normal bundle) is orientable for all $i$.  Assume that $\dim(H^*_G(X))=\dim(H^*_G(Z_i)$ for all $i$, and if $z \not \in Z$, then $\dim(H^*_{G_z}) < \dim(H^*_G(Z_i))$ for all $i$. Then, $$\deg(H^*_G(X)) = \sum_{i=1}^n deg(H^*_G(Z_i)).$$
\end{theorem}

Embedding $G$ in a unitary group $U$, and taking $F = U/S$ as usual, Lynn shows that for any $[A,c] \in \mathcal{Q'}_{max}(G,X)$ the sub-space $G \cdot (c \times F^A)$ satisfies the hypotheses required of  $Z_i$ in the above theorem, so deduces:
\begin{corollary}
For $X$ a compact, smooth manifold with $G$ a compact Lie group acting smoothly on $X$, $\deg(H^*_G(X)) = \sum_{[A,c] \in \mathcal{Q'}_{max}(G,X)} \deg(H^*_G(G\cdot (c \times F^A))$
\end{corollary}

Lynn argues by descent, which simply means using \ref{lemma Q formula}, and taking $X$ to be a point, to derive the following additivity formula for degree.

\begin{theorem}\cite{Lynn}
Let $G$ be a compact Lie group, and let $\mathcal{Q'}_{max}(G)$ be the set of conjugacy classes of maximal rank elementary abelian $p$-groups of $G$. Then,
$$\deg(H^*_G) = \sum_{[A] \in \mathcal{Q'}_{max}(G)}\frac{1}{|W_G(A)|} \deg(H^*_{C_G(A)}).$$
\end{theorem}

As noted in the introduction, she doesn't prove such a formula for equivariant cohomology $H^*_G(X)$ where $X$ is any $G$-space, and in any case, one would only expect her methods to generalize in the case of a smooth $G$-space $X$.

\section{Main Theorem on Degree} 
As usual, we assume that $G$ is compact Lie, and it acts continuously on the Hausdorff space $X$, which is compact or paracompact with finite mod-$p$ cohomological dimension.  Fix a prime $p$, and take $k= \mathbb{F}_p$ to be the field of coefficients for cohomology.  

Before we start, we make the simple observation that the set of  primes for $H^*_G(X)$ as an $H_G(X)$-module is the set of prime ideals in $H_G(X)$, since $Ann_{H_G(X)}H^*_G(X) = \{0\}.$

We will make use of Duflot's localization result \ref{theorem Duflot localization}, so let's look a little more closely at the ring $H^*_{C_G(A,c)}(c)$ for a given pair $(A,c) \in \mathcal{Q}(G,X)$. A first observation is that $H^*_{C_G(A,c)}(c)$ may be thought of as a module in the category $\grmod(H_{C_G(A,c)}(c))$, or as a module in $\grmod(H_G(X))$. Its structure as a finitely generated, graded module over $H_G(X)$ comes from the restriction map $res^G_{C}:H^*_G(X) \rightarrow H^*_{C_G(A,c)}(c)$ - an application of theorem \ref{theorem Quillen finiteness 2}. The following lemma relates the minimal primes of $H^*_{C_G(A,c)}(c)$ as a module over these two rings.

\begin{lemma} \label{lemma minimal primes for HC}
Let $[A,c] \in \mathcal{Q'}(G,X)$. Let $R=H_G(X)$, $S=H_{C_G(A,c)}(c)$, and make $S$-modules into  $R$-modules via the restriction map $res^G_C: R \rightarrow S$, so that $H^*_{C_G(A,c)}(c)$ is an element of $\grmod(R)$ and $\grmod(S)$. Then, 
\begin{itemize}
\item[i.] $\mathcal{Q'}(C_G(A,c),c) = \{ [A,c] \}$
\item[ii.] $\p^C \doteq \ker \left( H_{C_G(A,c)}(c) \rightarrow H_A/ \sqrt{0} \right)$ is the unique minimal prime for $H^*_{C_G(A,c)}(c)$ as an $S$-module.
\item[iii.] $(res^G_C)^{-1} \left( \p^C_{} \right) = \p,$ where $\p \doteq \ker \left( H_G(X) \rightarrow	H_A/\sqrt{0} \right)$.
\item[iv.] $\p$ is the unique minimal prime for $H^*_{C_G(A,c)}(c)$ as an $R$-module.
\item[v.] $*\ell_{S_{[\p^C]}}\left( H^*_{C_G(A,c)}(c)_{[\p^C]} \right) < \infty$ and $*\ell_{R_{[\p]}}\left( H^*_{C_G(A,c)}(c)_{[\p]} \right) < \infty$; also, 
\begin{align*}
\deg\left( H^*_{C_G(A,c)}(c) \right) &= *\ell_{S_{[\p^C]}}\left( H^*_{C_G(A,c)}(c)_{[\p^C]} \right) \deg(S/\p^C)\\
&=*\ell_{R_{[\p]}}\left( H^*_{C_G(A,c)}(c)_{[\p]} \right) \deg(R/\p).
\end{align*}
\end{itemize}
\end{lemma}

\begin{proof}  Write $C \doteq C_G(A,c)$.  Of course $A \subseteq C$ since $A$ is abelian.  Since $A$ acts trivially on $c$, we see that, for every $x \in c$, $A \subseteq C_x \subseteq G_x$.  We know that $A$ is a maximal elementary abelian subgroup of $G_x$, so $A$ is a maximal elementary abelian subgroup of $C_x$ as well.  Thus $[A,c] \in \mathcal{Q'}(C,c)$.  On the other hand, suppose that $[B,d] \in \mathcal{Q'}(C,c)$.   Then, $B$ is a maximal elementary abelian subgroup of $C_x$ for every $x \in d$, which is a component of $ c^B.$ Let $D$ be the subgroup of $C$ generated by $A$ and $B$.  Then, $D$ is an elementary abelian subgroup of $C$, by definition of $C$.   Suppose that $x \in d \subseteq c^B$.  Since $A$ fixes $x$ and $B$ fixes $x$, so does $D$, so that $D \subseteq C_x$.  By maximality of $B$ in $C_x$, $B = D$, so that $A \subseteq B$, and by maximality of $A$ in $G_x$, $A = B$, so $c^B = c$ and  $d = c$.  Thus, $(A,c) = (B,d)$ as pairs.  This proves i), and ii) follows from \ref{theorem Quillen main 2}.

To see iii), note that he following diagram is commutative, where $pt$ is any point in $c$: 

\begin{equation*}
\xymatrix{
(G,X)  & (C_G(A,c),c) \ar@{_{(}->}[l] & \\
(A,pt) \ar@{_{(}->}[u] \ar@{^{(}->}[ur]
}
\end{equation*}

By functoriality, there is a commutative diagram: 

\begin{equation*}
\xymatrix{
 H_G(X) \ar[r]^{res^G_{C}} \ar[d]_{\pi \circ res^G_A} & H_{C_G(A,c)}(c) \ar[dl]^{\pi \circ res^{C}_A} \\
 H_A(pt)/\sqrt{0} \\ 
}
\end{equation*}

Now, $\p$ is a minimal prime in $H_G(X)$ by \ref{theorem Quillen main 2}, so by commutativity of the above diagram, $\p = res^G_C\phantom{}^{-1}(\p^C_{})$. 

To see iv),  as noted above, $\p$ is a minimal prime for $H_G(X)$ and thus for $H^*_G(X)$ as an $R$-module. Let's show that $\p$ is a minimal prime for $H^*_{C_G(A,c)}(c)$ as an $R$-module.  We need to show that $\p$ is minimal over $Ann_R(H^*_{C_G(A,c)}(c)) \doteq \left \lbrace r \in R : res^G_C(r) \cdot x =0, \text{ for all } x \in H^*_{C_G(A,c)}(c) \right \rbrace$. But $H^*_{C_G(A,c)}(c)$ is a unital ring, so   $Ann_R(H^*_{C_G(A,c)}(c))= \ker(res^G_C)$, and by commutativity of the diagram $\ker(res^G_C) \subseteq \p$.  Since $\p$ is minimal in $R$, it must be minimal over $Ann_R(H^*_{C_G(A,c)}(c))$ and thus is a minimal prime for $H^*_{C_G(A,c)}(c)$ as an $R$-module. 

Finally, we need to show that $\p$ is the unique minimal prime for $H^*_{C_G(A,c)}(c)$ as an $R$-module.  Since we know that $res^G_C: R \rightarrow S$ is an integral extension, we may use  ``lying over".   Let $\q \subseteq R$ be another minimal prime for $H^*_{C_G(A,c)}(c)$ as an $R$-module; ``lying over" implies there exists a $\tilde{\q} \in Spec(S)$ such that $(res^G_C)^{-1}(\tilde{\q}) = \q$. But, $\p^C$ is the only minimal prime in $Spec(S)$, so $\p^C\subseteq \tilde{\q}.$ Thus, $(res^G_C)^{-1}(\p^C) \subseteq (res^G_C)^{-1}(\tilde{\q})$ which implies that $\p \subseteq \q$, but by assumption $\q$ is minimal, and therefore $\p=\q$. 

For the last proof,  set $N\doteq H^*_{C_G(A,c)}(c)$.  Parts ii) and iv) of this lemma show that $\p$ is a minimal prime for $N$ as an $R$-module and $\p^C$ is a minimal prime for $N$ as an $S$-module. Using  \ref{finite length lemma}, $N_{[\p]}$ is a *Artinian $R_{[\p]}$-module, and $N_{[\p^C]}$ is a *Artinian $S_{[\p^C]}$-module, which proves the claim on finite *length. 

To prove the claim on degree, we use the algebraic additivity formula for degree, Theorem \ref{theorem sum formula for degree and multiplicity}. Since we've shown in parts ii) and iv), that $\p$ is the unique minimal prime for $N$ as an $R$-module, and $\p^C$ is the unique minimal prime for $N$ as an $S$-module, there is only one summand in the algebraic additivity formula for degree, whether we consider $N$ as an $R$-module or an $S$-module.  Therefore, thinking of $N$ as an $R$-module, $\deg(N) = *\ell_{R_{[\p]}}( N_{[\p]}) \deg(R/\p)$; similarly the algebraic additivity formula applied to $N$ as an $S$-module gives $\deg(N) = *\ell_{S_{[\p^C]}}( N_{[\p^C]}) \deg(S/\p^C)$. 
\end{proof}

We state the following theorem in a general algebraic setting, and will demonstrate an application (Theorem \ref{theorem length of W action}) to the cohomology ring $H^*_{C_G(A,c)}(c)^{W_G(A,c)}$.  

\begin{theorem} \label{theorem decomp of cohomology by W}
Let $L$ be a $\mathbb{Z}$-graded $k$-algebra for a field $k$ (concentrated in degree $0$) and let $P$ be a graded $L$-module. Suppose the following:
\begin{itemize}
\item $P$ is a direct sum of graded $L$-submodules indexed by a finite set $\pi_0$:  $P = \oplus_{c \in \pi_0} P_c.$
\item $W$ is a finite group which acts as a group of graded $L$-module automorphisms on $P$. 
\item $W$ acts freely on the set $\pi_0$ and this action is compatible with the $W$-action on $P$ and the direct sum decomposition:   to be precise, for any $x \in P_{c}$, $w \cdot x \in P_{w\cdot c}$.
\end{itemize}

Then,
\begin{itemize}
\item[i.] Let $c_1, \ldots, c_t \in \pi_0$ be a set of representatives for the orbits of $W$ on $\pi_0$. Note that $t= |\pi_0|/|W|$. For each $j$ from $1$ to $t$, define $P[j]= \oplus_{w \in W} P_{wc_j}$. Then, $P[j]$ is a graded submodule of $P$ with respect to both $L$ and $k[W]$. Also, $P = \oplus_{j=1}^t P[j]$.  Here, $k[W]$ is regarded as a graded ring concentrated in degree $0$.
\item[ii.] $P$ is a free $k[W]$-module. 
\item[iii.] $P^W$ is isomorphic as a graded $L$-module to $\oplus_{j=1}^t P_{c_j}$.
\item[iv.] If  $P$ is a *Artinian $L$-module, then so is $P^W $ and, $$*\ell_{L}(P)=|W|*\ell_L(P^W).$$
\end{itemize}
\end{theorem}

\begin{proof}  For ease of notation, we assume that the direct sum decomposition is an internal direct sum decomposition, except in item iii).

Property i) is more or less by definition, using the hypotheses:    since the action of $W$ on $\pi_0$ is free,  $P_{wc_j} \cap P_{\tilde{w}c_j} = \{0\}$ if $w \neq \tilde{w}$, so $P[j]$ is a direct sum of graded $L$-submodules, as desired.  $P[j]$ is a $k[W]$-module by hypothesis, and its definition. Using the given decomposition $P=\oplus_c P_c$ along with the fact that we've picked $c_1, \ldots, c_t$ as representatives for the orbits under $W$ (and hence $P[j] \cap P[i] = 0$ for $i \neq j$) gives us the decomposition by orbits: $P=\oplus_{j=1}^t P[j]$. 

Property ii) is straightforward, but we give details:   let $E_d$ be a $k$-vector space basis for the degree $d$ homogeneous component of $P_{c_j}$. We claim that $E_d$ is a $k[W]$-basis for the degree $d$ component of $P[j]$. By picking a $k$-basis for every homogeneous component, and every $j$, we get a $k[W]$-basis  consisting of homogeneous elements for $P= \oplus_{j=1}^tP[j]$ (note that the basis may not be finite).

Writing out the details to see how the hypotheses are used, let $x \in P[j]_d$, and write  $x= \sum_{w \in W} x_w$, $x_w \in P_{w\cdot c_j}$. By hypothesis, $w^{-1}x_w \in P_{w^{-1}wc_j}=P_{c_j}$ for each $w$. Use the vector space basis of $P_{c_j}$ to write $w^{-1}x_w = \sum_{e \in E_d} \alpha_e(w)e$ where $\alpha_e(w) \in k$ for every $e \in E_d$ and equal to 0 for almost all $e$. Thus, $x = \sum_w x_w = \sum_w \sum_e \alpha_e(w) w e$, which is in the $k[W]$-span of the $E_d$.  For linear independence, suppose $\sum_e \xi_e e = 0$ for $\xi_e \in k[W]$; again $\xi_e= 0$ for almost all $e$. For each $e$, write $\xi_e$ as a linear combination $\sum_{w \in W} \alpha_e(w)w$, $\alpha_e(w) \in k$ for every $e$. Then, $\sum_e \xi_e e = \sum_e \sum_w \alpha_e(w) w e = \sum_w (\sum_e \alpha_e(w) we)$.  Now for each $w \in W$, $w\cdot E \doteq \{w\cdot e \mid e \in E_d \}$ is contained in $P_{wc_j}$, and $w \cdot E$ is a vector space basis for the degree $d$ homogeneous component of $P_{w c_j}$ because $w$ is a graded automorphism of $P$.  Therefore $\sum_e \alpha_e(w) w \cdot e = 0$, for every $w$,  using the direct sum decompostion; in turn, $ \alpha_e(w) = 0$, for every $w$ and every $e$, so $\xi_e= 0$ for every $e$. 

For property iii), define $\theta_j: P_{c_j} \rightarrow P[j]$ by $x \mapsto \sum_{w \in W} w\cdot x$.  Since each $w$ acts as a graded $L$-module homomorphism,  $\theta_j$ is a graded $L$-module homomorphism.  Further, $\theta_j$ is injective since $P[j]$ is direct sum. 

We claim that the image of $\theta_j$ is $P[j]^W$; in this case, the result is obtained since invariants distribute over direct sums of $k[W]$-modules. i.e. $\oplus_j \theta_j$ is an isomorphism from  $\oplus_{j=1}^t P_{c_j}$ to $P^W$. It's straightforward to see that the image of $\theta_j$ is contained in $P[j]^W$.
On the other hand, let $x \in P[j]^W$, and write $x$ uniquely as $\sum_{w \in W} x_w$, where $x_w \in P_{wc_j}$ for every $w$. Now, for every $g \in W$, $gx = \sum_{w \in W} gx_w = \sum_{w \in W} x_w = x$. Since $gx_w \in P_{gwc_j}$, using the direct sum property, if $1$ is the identity in the group $W$,   $w x_{1} = x_w$ for every  $w$ and  $\theta_j(x_{1})=\sum_{w \in W} wx_{1} = x$.

For iv), as expected, *length is additive over short exact sequences of $L$-modules and thus over direct sums.  Of course, for each $w$ and each $c_j$, $*\ell_L(P_{cj})=*\ell_L(P_{wc_j})$ since $W$ acts as a group of $L$-automorphisms. Thus, $*\ell_L(P)=\sum_{j=1}^t \sum_{w \in W} *\ell_L(P_{wc_j})=|W|\sum_{j=1}^t*\ell_L(P_{c_j})=|W|*\ell_L(P^W)$.

\end{proof}

In the following theorem, we've fixed an embedding of $G$ in a unitary group $U$, and $F \doteq U/S$ is the $G$-space described previously;  if $(A,c) \in \mathcal{Q}(G,X)$,  we've also already explained how we are considering $H^*_{C_G(A,c)}(c\times F^A)$ as a graded $H_G(X)$-module.

\begin{theorem} \label{theorem length of W action} 
Let $[A,c] \in \mathcal{Q'}(G,X)$ and $\p \doteq \p_{(A,c)}$. 
\begin{itemize}
\item[i.] Let $\pi_0$ be the set of connected components of $c \times F^A$. Then,
\begin{itemize}
\item[(a)] $H^*_{C_G(A,c)}(c\times F^A)$ is a free $k[W_G(A,c)]$-module.
\item[(b)] If $c_i$, $1 \leq i \leq t$, are the representatives for the orbits of $W_G(A,c)$ acting on the set of components of $c \times F^A$, then  $$H^*_{C_G(A,c)}(c\times F^A)^{W_G(A,c)} \cong \oplus_{i=1}^t H^*_{C_G(A,c)}(c_i),$$  as graded $H_G(X)$-modules.
\end{itemize}
\item[ii.]  $*\ell_{H_G(X)_{[\p]}} \left( H^*_{C_G(A,c)}(c \times F^A)^{W_G(A,c)}_{[\p]} \right) = \frac{1}{|W_G(A,c)|} *\ell_{H_G(X)_{[\p]}}\left( H^*_{C_G(A,c)}(c \times F^A)_{[\p]} \right)$
\end{itemize}
\end{theorem}

\begin{proof} Set $R \doteq H_G(X)$, $W \doteq W_G(A,c)$.
Lemma \ref{lemma W acts freely} tells us that $W$ acts freely on the set $\pi_0$ of connected components of $c \times F^A$, and that this set of components is a finite set.   Quillen shows that $W$ is a finite group (\cite{Quillen2}).

Using lemma \ref{lemma W acts freely}, each $d \in \pi_0$ is a $C_G(A,c)$-space, so that  $H^*_{C_G(A,c)}(c \times F^A) = \oplus_{d \in \pi_0} H^*_{C_G(A,c)}(d)$; this is an  isomorphism of graded $R$-modules.  Now, the action of $W$ on $H^q_{C_G(A,c)}(c \times F^A)$, for any $q$,  comes from the natural geometric action of $N_G(A,c)$ on $c \times F^A$ (see the discussion on page 4) and thus takes components to components, so for every $x \in H^q_{C_G(A,c)}(d)$, and every $n \in N_G(A,c)$, $n^*(x) \in H^q_{C_G(A,c)}(n\cdot d).$  Thus, the same is true for the $W$-action on cohomology;  the discussion on pages 4 and 5 tells us that this $W$-action on cohomology gives compatibility with the $R$-module structure on $H^q_{C_G(A,c)}(c \times F^A)$.  Putting this together, we see that the hypotheses required for conclusions i), ii), and iii) of  Theorem \ref{theorem decomp of cohomology by W}  are true, for $L = R$, $W = W_G(A,c)$, $P =  H^*_{C_G(A,c)}(c \times F^A)$ and $\pi_0$ as in this theorem.  Therefore, i.a) and i.b) are true.

Finally, using \ref{theorem Quillen main 2}, $\p$ is a minimal prime in $R$, so, using \ref{finite length lemma},  the graded localization of $H^*_{C_G(A,c)}(c \times F^A)$ at the minimal prime $\p$ is a *Artinian $R_{[\p]}$-module.  Now, note that $W$ acts on $ H^*_{C_G(A,c)}(c \times F^A)_{[\p]}$ via the formula $w \cdot (a/b) = (w \cdot a)/b$.   This is well-defined since $W$ acts as a group of $R$-module isomorphisms (see the discussion on pages 4 and 5) and the denominators all come from $R$. Since we've already verified the hypotheses of Theorem \ref{theorem decomp of cohomology by W} for $L=R$, $ P = H^*_{C_G(A,c)}(c \times F^A)$, and these verifications are well-behaved with respect to localization, the hypotheses of Theorem \ref{theorem decomp of cohomology by W} are also true for $L = R_{[\p]}$ and $P = H^*_{C_G(A,c)}(c \times F^A)_{[\p]}$ ; applying part iv) of \ref{theorem decomp of cohomology by W} to the localized set-up, we are done. \end{proof}

\begin{theorem} \label{theorem length G to C}
Let $R=H_G(X)$.  Suppose that $[A,c] \in \mathcal{Q'}(G,X) $ and $\p \doteq \p_{(A,c)}$.  Then, $$*\ell_{R_{[\p]}}(H^*_G(X)_{[\p]}) = \frac{1}{|W_G(A,c)|} *\ell_{R_{[\p]}}(H^*_{C_G(A,c)}(c)_{[\p]}).$$
\end{theorem}

\begin{proof}
Using (\ref{theorem Quillen main 2}) we know that $\p$ is a minimal prime ideal in $R$.

In this computation abbreviate $C_G(A,c)$ to $C$, $W_G(A,c)$ to $W$. Note the following facts.

1) We use lemma \ref{lemma Q formula}, computing $H^*_G(X \times F) \cong H^*_G(X) \otimes_{k} H^*(F)$, where the $R$-module structure is given by $r(a \otimes v) \doteq ra \otimes v$. Therefore, when we localize, $$H^*_G(X \times F)_{[\p]} \cong H^*_G(X)_{[\p]} \otimes_{k} H^*(F)$$ as $R_{[\p]}$-modules.

Now, $F$ is a finite dimensional compact manifold, so $H^*(F)$ is a finite dimensional graded vector space over $k$, let's say it has dimension $m$; thus, using lemma \ref{length of a tensor product}, $$*\ell_{R_{[\p]}}(H^*_G(X \times F)_{[\p]})= *\ell_{R_{[\p]}}(H^*_G(X)_{[\p]})\cdot m.$$

2) Refer back to the diagram of theorem \ref{theorem Duflot diagram}; that theorem asserts that $H^*_G(X \times F)_{[\p]} \cong H^*_{C}(c \times F^A)^W_{[\p]}$ as $R_{[\p]}$-modules.  

3) Using theorem \ref{theorem length of W action},  $$*\ell_{R_{[\p]}} \left( H^*_{C}(c \times F^A)^{W}_{[\p]} \right) = \frac{1}{|W|} *\ell_{R_{[\p]}}\left( H^*_{C}(c \times F^A)_{[\p]} \right).$$

4) By referring again to the diagram of theorem \ref{theorem Duflot diagram}, the fourth vertical arrow gives an isomomorphism: $$H^*_{C}(c \times F^A)_{[\p]} \cong H^*_{C}(c \times F)_{[\p]}$$ as $R_{[\p]}$-modules.

5) Again using lemma \ref{lemma Q formula} and the properties of localization for the case of the pair $(C,c)$, $H^*_C(c \times F)_{[\p]} \cong H^*_{C}(c)_{[\p]} \otimes_{k} H^*(F)$ as graded $R_{[\p]}$-modules.

Putting these five results together gives us the following computation: 
\begin{align*}
*\ell_{R_{[\p]}}(H^*_G(X)_{[\p]})&\stackrel{(1)}{=} \frac{1}{m}*\ell_{R_{[\p]}}(H^*_G(X \times F)_{[\p]})\\
&\stackrel{(2)}{=} \frac{1}{m}*\ell_{R_{[\p]}}(H^*_{C}(c \times F^A)^W_{[\p]}))\\
&\stackrel{(3)}{=} \frac{1}{m}\frac{1}{|W|} *\ell_{R_{[\p]}}(H^*_{C}(c \times F^A)_{[\p]}) \\
&\stackrel{(4)}{=}  \frac{1}{m}\frac{1}{|W|} *\ell_{R_{[\p]}}(H^*_{C}(c \times F)_{[\p]}) \\
&\stackrel{(5)}{=}  \frac{1}{m}\frac{1}{|W|}\cdot m \cdot *\ell_{R_{[\p]}}(H^*_{C}(c)_{[\p]})\\
&{=} \frac{1}{|W|} *\ell_{R_{[\p]}}(H^*_{C}(c)_{[\p]})\\
\end{align*}
\end{proof}

We are now in position to prove our main result of the paper.

\begin{theorem} \label{theorem main thm geometric degree}
Suppose $X$ is a $G$-space.  Then,
$$\deg(H^*_G(X)) = \sum_{[A,c] \in \mathcal{Q'}_{max}(G,X)}\frac{1}{|W_G(A,c)|} \deg(H^*_{C_G(A,c)}(c)).$$
Furthermore, recall the algebraic additivity formula for degree (theorem \ref{theorem sum formula for degree and multiplicity}) for a graded module $M \in \grmod(R)$: $$\deg(M) = \sum_{\p \in \mathcal{D}(M)}*\ell_{R_{[\p]}}(M_{[\p]})\cdot \deg(R/\p).$$  Then, for $M =H^*_G(X)$, and $R=H_G(X)$, the index sets of the two additivity formulae are in 1-1 correspondence and the summands are equal term-by-term under this correspondence.  Thus, the result can be understood as a ``geometric" interpretation of the algebraic additivity formula. 
\end{theorem}

\begin{proof}
We use the following notation: $M=H^*_G(X)$, $R=H_G(X)$, $N=H^*_{C_G(A,c)}(c)$, and $S=H_{C_G(A,c)}(c)$. 

Let's apply the algebraic additivity formula to $M$ as an $R$-module (Theorem \ref{theorem sum formula for degree and multiplicity} )  $$\deg(M) {=} \sum_{\p \in \mathcal{D}(M)}*\ell_{R_{[\p]}}(M_{[\p]})\cdot \deg(R/\p).$$ Recall that $\mathcal{D}(M) \doteq \{ \q \in Spec(R): *\dim(R/\q) = *\dim_R(M) \}$. By Quillen's  theorem \ref{theorem Quillen main 2}, there is a bijective correspondence between $\mathcal{D}(M)$ and $\mathcal{Q}'_{max}(G,X)$. 

Now, the *length factor in each summand of the degree formula may be re-written using the previous theorem (\ref{theorem length G to C}): $$*\ell_{R_{[\p]}}(M_{[\p]}) = \frac{1}{|W|}*\ell_{R_{[\p]}}(N_{[\p]}).$$

Let $[A,c]$ be any element of $\mathcal{Q}'_{max}(G,X)$, and consider the corresponding primes \newline $\p \doteq \ker\left( res^G_A: H^*_G(X) \rightarrow H_A/ \sqrt{0} \right)$, and $\p^C \doteq \ker\left( res^G_A: H^*_{C_G(A,c)}(c) \rightarrow H_A/ \sqrt{0} \right)$.  In Lemma \ref{lemma minimal primes for HC}, we compared the sum formula for the degree of $N$ given its $R$ and $S$-module structures. In particular, we showed that, in either case, the additivity formula for  the degree of $N$ has only one summand,  which may be computed using either module structure:
\begin{align*}
\deg( N ) &= *\ell_{S_{[\p^C]}}\left( N_{[\p^C]} \right) \deg(S/\p^C)\\
&=*\ell_{R_{[\p]}}\left( N_{[\p]} \right) \deg(R/\p).
\end{align*}

We therefore have the computation:
\begin{align*}
\deg(M)&=\sum_{\p \in \mathcal{D}(M)}*\ell_{R_{[\p]}}(M_{[\p]})\cdot \deg(R/ \p)\\
&= \sum_{[A,c] \in \mathcal{Q}'_{max}(G,X)} \frac{1}{|W_G(A,c)|} *\ell_{R_{[\p]}}(N_{[\p]})\deg(R/ \p)\\
&= \sum_{[A,c] \in \mathcal{Q}'_{max}(G,X)} \frac{1}{|W_G(A,c)|}*\ell_{S_{[\p^C]}}(N_{[\p^C]})\deg(S/\p^C)\\
&= \sum_{[A,c] \in \mathcal{Q}'_{max}(G,X)}\frac{1}{|W_G(A,c)|}\deg(N).
\end{align*}

\end{proof}

\end{document}